\documentclass[a4paper]{article}
\usepackage{amsthm}
\usepackage{hyperref}
\hypersetup{colorlinks=true,linkcolor=blue,citecolor=red,linktocpage=true}
\newtheorem{proposition}{Proposition}[section]
\newtheorem{lemma}[proposition]{Lemma}
\newtheorem{theorem}[proposition]{Theorem}
\newtheorem{corollary}[proposition]{Corollary}

\newtheorem{THM}{Theorem}
\usepackage[utf8]{inputenc}
\theoremstyle{definition}
\newtheorem{definition}[proposition]{Definition}
 \newtheorem{example}[proposition]{Example}
 \newtheorem{remark}[proposition]{Remark}
\usepackage[utf8]{inputenc}
\usepackage[english]{babel}
\usepackage{amsmath}
\usepackage{amssymb}
\usepackage{textcomp}
\usepackage{hyphenat}
\usepackage[all]{xy}
\usepackage{graphicx}
\usepackage{float}
\usepackage{listofsymbols}
\usepackage{tikz}
\usepackage{circuitikz}
\usepackage{tikz-cd}

\usepackage{latexsym}
\usepackage{epsfig}

\newcommand{\C}{\mathbb C}

\usepackage{nomencl}
\makenomenclature

\emergencystretch=\maxdimen
\begin{document}
\title{Topology of leaves of generic logarithmic foliations on $\mathbb{CP}^2$.}
\author{Diego Rodr\'{i}guez Guzm\'{a}n\footnote{The research for this paper is based on the author's thesis presented for the degree of Doctor in Mathematics at IMPA, July 2016. He was partially supported by CONACYT-FORDECYT 265667 and CONACYT CB2016/286447 to conclude this paper.}}
\date{August 2019}

\maketitle
%\begin{center}\rule{0.9\textwidth}{0.1mm} \end{center}

\begin{abstract}
We describe the topological types of leaves of generic logarithmic foliations on the complex projective plane. We prove that all leaves, except for a finite many are biholomorphic to $\mathbb{C}$ or homeomorphic to the surface known as  Loch Ness monster.
\end{abstract}

\begin{quote}

\end{quote}

\begin{quote}
\end{quote}

\section{Introduction}
A leaf $\mathcal{L}$ of a holomorphic foliation $\mathcal{F}$ by complex curves on a complex manifold $M$ is a Riemann surface. In general, $\mathcal{L}$ is non-compact, and its closure $\overline{\mathcal{L}}$ in $M$ does not correspond to a compact Riemann surface embedding in $M$. Given condition on $M$ and $\mathcal{F}$, an interesting question is: 
what is the topological type of the leaves of $\mathcal{F}$?

For the case of orientable regular $\mathcal{C}^r$ differentiable foliation $\mathcal{F}$ of real dimension 2 with $r\geq3$ on compact manifolds $M$, a theorem of E. Ghys \cite{ghys1995topologie} proves that if  $\mathcal{F}$ has no compact leaf, then almost every leaf  has one of the following six topological types: a plane, a cylinder, a sphere without a Cantor set, a plane with infinitely many handles attached, a cylinder with infinitely many handles attached to both ends and a sphere minus a Cantor set with infinitely many handles attached to every end. In such case $\mathcal{F}$ has a compact leaf $C$ with finite holonomy group $\mathrm{Hol}(C,\mathcal{F})$, the local stability theorem of Reeb implies that there is a open neighborhood $U\subset M$ of $C$ saturated by $\mathcal{F}$ and each leaf in $U$ is a finite cover of $C$.

In the holomorphic case, the foliations have singularities; this does not allow to extend the above results. However, Anosov conjectured that for a generic holomorphic foliation $\mathcal{F}$ on the complex projective plane $\mathbb{CP}^2$ all the leaves are disks, except for a countably set of leaves which are topological cylinders. This conjecture remains unsolved. In a recent work N. Sibony and E. F. Wold \cite{sibony2017topology} give conditions on holomorphic foliations $\mathcal{F}$ on complex compact surfaces $X$, which guarantees the existence of a non-trivial closed subset $Y$ of $X$ saturated by $\mathcal{F}$, such that every leaf in $Y$ is a disk.

Our purpose is to give the topological type of the leaves for generic logarithmic foliations on $\mathbb{CP}^2$. The logarithmic foliation on complex surfaces lets invariant compact Riemann surfaces. Thus, we study
when a foliation $\mathcal{F}$ by real surfaces let invariant a compact Riemann surface$\Sigma_g$ with only $n$ singular points $\{p_j\}_{j=1,\ldots,n}$ of $\mathcal{F}$, in particular, $\Sigma_{g,n}=\Sigma_g-\{p_j\}$ is a leaf of $\mathcal{F}$. We will prove the following statement.
\begin{THM}\label{THM:ReebType}
 Let $\mathcal{F}$ be a singular $\mathcal{C}^r$ differential foliation of real dimension 2 on a manifold $M$ and $o\in\Sigma_{g,n}$ a regular point of $\mathcal{F}$, with $r\geq1$. If $\mathrm{Hol}(\Sigma_{g,n},\mathcal{F})$ is an infinite group, then for each $N\in\mathbb{N}$ there is an embedding
 \[
    \varepsilon:B_{N,\delta}(\tilde{o})\rightarrow(\mathcal{L}_p,p),
 \]
of the closed subset $B_{N,\delta}(\tilde{o})$ of the regular cover $\Sigma_{g,n}^H$ of $\Sigma_{g,n}$ corresponding to the normal subgroup $H=\ker(\mathrm{Hol}(\Sigma_{g,n},\mathcal{F}))$, and $\mathcal{L}_p$ is a leaf  through a regular point $p\in \tau-{o}$, with $\tau$ the germ of transversal to $\mathcal{F}$ at $o$ to define $\mathrm{Hol}(\Sigma_{g,n},\mathcal{F})$.
\end{THM}
The subset $B_{N,\delta}(\tilde{o})$ is a finite cover of $\Sigma_{g,n}-\cup_j B_{\delta}(p_j)$ and carry on with some topology of $\Sigma_{g,n}^H$, with $B_{\delta}(p_j)$ open neighborhoods of $p_j$ in $\Sigma_g$ (see Definition \ref{def:graphball}). Moreover, if the regular cover has infinitely many handles attached then, the above result implies that the foliation has leaves with infinitely many handles attached. In this sense, another Known result given by Goncharuk and Kudryashov \cite{goncharuk2014genera}, which proved that for generic holomorphic foliations $\mathcal{F}$ on $\mathbb{CP}^2$ with a projective line invariant by $\mathcal{F}$ has leaves with $n$ handles attached, where $n$ is a number depending on the degree of the polynomial vector fields defining the foliation.
In general, the holonomy group of an invariant compact Riemann surface by a logarithmic foliation on a complex surface is infinite, and the corresponding regular cover has infinitely many handles.  This claim is proved in Section 4 and let us prove our main result.
\begin{THM}\label{THM:LNMP2}
Let $\mathcal{F}$ be a logarithmic foliations on $\mathbb{CP}^2$ defined by a closed logarithmic 1-form $\omega$ with polar divisor $D$. Suppose that $D$ is normal crossing and the ratios $\lambda_j/\lambda_i$ of the residues of $\omega$ are not negative real numbers. Then all leaves, except for finite set of leaves, are homeomorphic to one of the following real surfaces:
\begin{itemize}
\item[a)]The plane, in this case is biholomorphic to $\mathbb{C}$.
\item[b)]The Loch Ness Monster,i.e., the real plane with infinitely many handles attached. 
\end{itemize}
\end{THM}
The proof of this theorem relies on the description of the following topological invariants of an open orientable real surface $\Sigma$ (see \cite{richards1963classification} for more details):
\begin{itemize}
\item[a)]The space of ends $\mathcal{E}(\Sigma)$, which is compact and totally disconnected.
\item[b)]The genus, which is the number of handles attached on $\Sigma$.
\end{itemize}
We recall these concepts in Section 2 and use them to show that there is only eleven infinite regular covers of compact Riemann surfaces minus a finite set of points. Also, we show that each of these surfaces is realizable as a generic leaf of a Riccati foliation.
Section 3 provides a detailed exposition of the proof of Theorem 1, and we give all the concepts related with it. In the Section 4, we will be interested with the description of the end space and genus of the leaves of generic logarithmic foliations to give the proof of Theorem 2.

\textbf{Acknowledgements.} The author thanks Jorge Pereira, for fruitful discussions and corrections. Also, he thanks the Universidad Aut\'{o}noma de Aguascalientes and the CIMAT for their hospitality during the corresponding stays.

\section{Regular covers of real surfaces}
This section presents a description of regular covers of orientable bordered real surfaces, via Cayley graphs. In particular, we show that any of these regular covers with infinite deck transformation group is of one of the following eleven topological types: 
\begin{itemize}\label{list:11}
\item[1)] the plane,
\item[2)] the Loch Ness monster, i.e., the real plane with infinitely many handles attached,
\item[3)] the cylinder,
\item[4)] the Jacob's ladder,  i.e., the cylinder with infinitely many handles attached to both directions,
\item[5)] the Cantor tree, i.e., the sphere without a Cantor set,
\item[6)] the blooming Cantor tree, i.e., the Cantor tree with infinitely many handles attached to each end,
\item[7)] the plane without an infinite discrete set,
\item[8)] the Loch Ness monster without an infinite discrete set,
\item[9)] the Jacob's ladder without an infinite discrete set
\item[10)] the Cantor tree without an infinite discrete set,
\item[11)] the blooming Cantor tree without an infinite discrete set.
\end{itemize}
Although this result seems to be well-known, we could only find  in the literature a proof for normal covers of compact Riemann surfaces \cite{goldman1967open}. We will see that each of these surfaces is realizable as the generic leaf of regular or singular foliations of real dimension two on compact manifolds. The properties and conclusions in this section support further development in this text.

Let $\Sigma_{g,n}$ denote the orientable real surface $\Sigma_g$ of genus $g$ whit a set $\{p_j\}\subset\Sigma_g$ of $n$ points taken out. We recall that every compact orientable surface $\Sigma_g$, with $g>0$, is constructed from a polygon $P_{4g}$ with $4g$-sides by identifying pairs of edges. A couple of edges identified will be labeled by the letter $a$ if the direction for attaching correspond to the chosen orientation of $\partial P_{4g}$ or $a^{-1}$ if it is counter the orientation.

We will think the surface $\Sigma_{g,n}$, with $g>0$ and $n\geq 0$, as a punctured polygon $P_{4g}^n\subset P_{4g}$ of $4g$ edges without $n$ points $\{p_1,\ldots,p_n\}$ of its interior and with boundary $a_{1}b_{1}a_{1}^{-1}b_{1}^{-1}\cdots a_{g}b_{g}a_{g}^{-1}b_{g}^{-1}$. Let $o\in \Sigma_{g,n}$ denote the vertices attached of $P_{4g}$. We can choose the generators of the fundamental group $\pi_1(\Sigma_{g,n},o)$ to be the homotopy class of \[\mathfrak{G}_{g,n}:=\{a_1,b_1,\ldots,a_g,b_g,c_1,\ldots,c_{n-1}\},\] where $a_j,b_k$ correspond to the edges of $P_{4g}^n$ and the curves $c_j$ ,with $j=1,\ldots,n$, are the boundaries $\partial B(p_j,\delta)$ for $0<\delta<<1$ such that the intersections $c_j\cap c_i$ and $c_j\cap\partial P_{4g}$ are empty for all $j$ and $j\neq i$. These generators of $\pi_1(\Sigma_{g,n},o)$  will be called \emph{canonical generators}. When $g=0$ we only consider the generators $c_j$. 

For every normal subgroup $H$ of $\pi_1(\Sigma_{g,n},o)$, there is a regular cover \[\rho_H:\Sigma_{g,n}^H\rightarrow\Sigma_{g,n}\] of $\Sigma_{g,n}$ such that the image $\rho_{H*}(\pi_1(\Sigma_{g,n}^H))$ is $H$ and its group of deck transformations $A_{g,n}^H$ is isomorphic to $\pi_1(\Sigma_{g,n},o)/H$ (see \cite[p.71]{hatcher2002algebraic} for more details). In particular, we have a epimorphism $\varrho_H:\pi_1(\Sigma_{g,n},o)\rightarrow A_{g,n}^H$ whose kernel is $H$.

Remind that the Cayley graph $\mathrm{Cayley}(A_{g,n}^H,\mathfrak{G}_{g,n})$, of $A_{g,n}^H$ respect $\mathfrak{G}_{g,n}$ is the graph whose set of vertices $V$ is $A_{g,n}^H$ and its set of edges is \[\{(a,\gamma\cdot a)\in V\times V|\gamma\in \mathfrak{G}_{g,n}\}.\]
In case some $\gamma\in\mathfrak{G}_{g,n}$ is in $H$ the edge $(a,\gamma\cdot a)$ is a loop.

Fix a point $\tilde{o}$ in $\rho_{H}^{-1}(o)$ and $\mathit{id}$ denotes the identity element of $A_{g,n}^H$. We give an embedding of $\mathrm{Cayley}(A_{g,n}^H,\mathfrak{G}_{g,n})$ in $\Sigma_{g,n}^H$ as follows: 
\[\begin{array}{lll}
\varepsilon:\mathrm{Cayley}(A_{g,n}^H,\mathfrak{G}_{g,n})&\rightarrow & \Sigma_{g,n}^H\\
\mathit{id} & \mapsto & \tilde{o} \\
a & \mapsto & a\tilde{o} \\
(a,\gamma a) & \mapsto & \tilde{\gamma}_{a\tilde{o}}
\end{array}\]
where $\tilde{\gamma}_{a\tilde{o}}$ is the lift of $\gamma\in\mathfrak{G}_{g,n}$ starting at $a\tilde{o}$ and end point $a_\gamma a\tilde{o}$, with $\varrho_{H}(\gamma)=a_\gamma$. From now on, $\mathrm{Cayley}(A_{g,n}^H)$ stands for the image of the embedding $\varepsilon$.

We consider that the distance $d(\tilde{o},\tilde{o}')$ between two vertices of $\mathrm{Cayley}(A_{g,n}^H)$ is the  length of the shortest $(\tilde{o},\tilde{o}')$-path, where the length of $(\tilde{o},\tilde{o}')$-path is the number of different edges which it contains. 

\begin{definition}\label{def:graphball} 
The ball $B_N(\tilde{o})$ of center $\tilde{o}$ and radius $N\in\mathbb{N}_{>0}$ in $\mathrm{Cayley}(A_{g,n}^H)$ is the collection of vertices of $\mathrm{Cayley}(A_{g,n}^H)$ at distance at most $N$ of $\tilde{o}$ and of paths of length less than $N$ starting at $\tilde{o}$. 
As $\Sigma_{g,n}=\Sigma_g-\{p_j\}_{j=1}^n$, we denote by $\Sigma_{g,n}(\delta)$ the compact surface, which is the complement of $\cup_{j=1}^n B(p_j,\delta)$ in $\Sigma_g$, and denote by $D_{g,n}(\delta)$ a fundamental domain of $\Sigma_{g,n}(\delta)$ such that $o\in D_{g,n}(\delta)$. Thus we define the $(N,\delta)$-ball $B_{N,\delta}(\tilde{o})\subset\Sigma_{g,n}^H$ as the closure of \[
\bigcup\limits_{\tilde{o}'\in B_N(\tilde{o})}\tilde{D}_{g,n}(\delta,\tilde{o}'),\]
where $\tilde{D}_{g,n}(\delta,\tilde{o}')$ is the lift of $\tilde{D}_{g,n}(\delta)$ at $\tilde{o}'$.
\end{definition}
Now, we recall that the topological classification of noncompact real surfaces states that the invariants which decide the topological class of an orientable connected real surface $S$ are its \textit{end space} $\mathcal{E}(S)$, genus finite or infinite, and the elements of $\mathcal{E}(S)$ where the genus accumulates in the latest case (see for instance\cite{richards1963classification}).
An end $e\in\mathcal{E}(S)$ is an equivalent class of a nested sequence $P_1\supset P_2\supset\cdots$ of open connected unbounded sets in $S$ satisfying the following conditions:
(i) the boundary of $P_k$ in $S$ is compact for all $k$;
(ii) the sequence $\{P_k\}$ does not have any common point.

The genus of $S$ can be understood as the limit of the genus of a compact exhaustion $\{K_i\}$ of $S$  such that the boundary of each $K_i$ is a union of closed simple curves.  We define the \emph{genus} $g(K_i)$ of $K_i$ as follows:

\[g(K_i)=\frac{1}{2}(2-\chi(K_i)-r_i),
\]
where $\chi(K_i)$ is the Euler characteristic  of $K_i$ and $r_i$ the number of boundary components of $K_i$ (see \cite{brahana21twodimensional}). Moreover, we will say that $e\in\mathcal{E}(S)$ accumulates genus if any sequence $\{P_k\}$ representing $e$, satisfies that the genus of $P_k$ is nonzero for all $k$. Otherwise, we say that $e$ is planar. Here $\mathcal{E}'(S)$ denotes the subset of ends which accumulates genus. The following results describe these invariants for $\Sigma_{g,n}^H$.

\begin{lemma}\label{lm:planarends}
Let $\Sigma_{g,n},\Sigma_{g,n}^H,H$ and $c_j\in\mathfrak{G}_{g,n}$ be as defined above, with $n>0$. If there exists an integer $m$ distinct from zero such that $c_j^m\in H$, then $\mathcal{E}(\Sigma_{g,n}^H)$ contains a subset of planar ends with discrete topology. Otherwise, if there is no $c_j$ with this property then $\mathcal{E}(\Sigma_{g,n}^H)$ is a single point.
\end{lemma}

\begin{proof}
Let $\tilde{c_j}$ be the lift of $c_j$ in $\Sigma_{g,n}^H$ trough a point $\tilde{p}$ in $\rho_H^{-1}(p)$, whit $p\in c_j$. Assume that $c_j^m$ is in $H$, where $m=\mathrm{min}\{b\in\mathbb{Z}_{>0}|c_j^b\in H\}$. Therefore $\tilde{c_j}^m$ is a closed curve and a finite cover of $c_j$, which is the boundary of the pointed disk $\mathbb{D}_{j,\delta}=B(p_j,\delta)-\{p_j\}\subset\Sigma_{g,n}$. The connected component $\widetilde{\mathbb{D}_{j,\delta}}$ of $\rho_H^{-1}(\mathbb{D}_{j,\delta})$, whose boundary contains $\tilde{p}$, is a finite cover of $\mathbb{D}_{j,\delta}$. Choose a decreasing sequence $\{\delta_i\}$ of positive numbers, with $\delta_0=\delta$, which converges to zero. The lifts $\widetilde{\mathbb{D}_{j,\delta_i}}$ of $\mathbb{D}_{j,\delta_i}$ contained in $\widetilde{\mathbb{D}_{j,\delta}}$ define a planar end of $\Sigma_{g,n}^H$. 
Analogously, each connected component of $\rho_H^{-1}(\mathbb{D}_{j,d})$ defines a planar end. Moreover, these define a discrete set of planar ends of $\Sigma_{g,n}^H$.

In the case that no cycle $c_{j}^{m}$ is in $H$, with $j\in\{1,\ldots,n\}$ and $m\in\mathbb{Z}-\{0\}$, we will prove that for any compact 
$K\subset\Sigma_{g,n}^H$, exists a compact $K'\supset K$ in $\Sigma_{g,n}^H$ such that $\Sigma_{g,n}^H-K'$ is connected. The latest affirmation implies that the end space $\mathcal{E}(\Sigma_{g,n}^H)$ is a single point. 
Choose $\delta,\tilde{o}\in\rho_{H}^{-1}(o)$, and $N\in\mathbb{N}_{>0}$ such that $K\cap \tilde{D}_{g,n}(\delta,\tilde{o})$ is non-empty and $\rho_H(K)\subset\Sigma_{g,n}(\delta)$  and $K\subset B_{N,\delta}(\tilde{o})$. By assumption, each connected component of $\rho^{-1}(c_j)$ is an unbounded curve, with  $c_j=\partial B_{\delta/2}(p_j)$, and its intersection with $K'=B_{N,\delta}(\tilde{o})$ is empty. Thus, we can follow some unbounded components of $\rho^{-1}(c_j)$ to join any two points in $\Sigma_{g,n}^H-K'$, which is the desire conclusion.  
\end{proof}

\begin{lemma}\label{lm:ends'}
For any infinite regular cover $\Sigma_{g,n}^H$ the set $\mathcal{E}'(\Sigma_{g,n}^H)$ is one of the following sets: an empty set, a single point, two points or a cantor set.
\end{lemma}
\begin{proof}
Let $J\subset\{1,\ldots,n\}$ such that $c_j$ satisfies that $c_j^{m_j}\in H$, where $m_j\in\mathbb{Z}_{>0}$ is defined as in the proof above. Plug the holes of $\Sigma_{g,n}^H$ surrounded by the connected components of $\rho_{H}^{-1}(c_j)$, with $j\in J$. This gives us a new surface $\Sigma_{g,n}^H(J)$ which contains $\Sigma_{g,n}^H$. Consider the subgroup $G$ of $A_{g,n}^H$ generated by $\varrho_{H}(\mathfrak{G}_{g,n}-\{c_j\}_{j\in J})$.
If $J=\{1,...,n\}$, then the group $G$ acts properly discontinuous on $\Sigma_{g,n}^H(J)$, and it is cocompact. Thus, the end space of $\Sigma_{g,n}^H(J)$ coincides with the end space of $G$ (see for instance \cite{scott_wall_wall_1979}), which by the \cite[Theorem8.2.11]{loh2011geometric} and \cite[Theorem 15.2]{goldman1967open} can only be empty, a single point, two points or a Cantor set. In the case $J$ is a proper subset of $\{1,..., n\}$, Lemma \ref{lm:planarends} implies that the end set of $\Sigma_{g,n}^H(J)$ is a point. 
 When $\mathcal{E}'(\Sigma_{g,n}^H(J))$ is non-empty, the group action of $G$ on $\Sigma_{g,n}^H(J)$ implies that $\mathcal{E}'(\Sigma_{g,n}^H(J))=\mathcal{E}(\Sigma_{g,n}^H(J))$. Thus, we have $\mathcal{E}'(\Sigma_{g,n}^H)=\mathcal{E}'(\Sigma_{g,n}^H(J))$, which completes the proof.
\end{proof}

The Riemann-Hurwitz formula gives the topological type of any finite regular cover $\Sigma_{g,n}^H$ of $\Sigma_{g,n}$. The Lemmas \ref{lm:planarends} and \ref{lm:ends'} well describe the invariants $\mathcal{E}(\Sigma_{g,n}^H)$ and $\mathcal{E}'(\Sigma_{g,n}^H)$, which imply the following result.

\begin{theorem}\label{THM:covers}
 Let $\Sigma_{g,n}^H$ be an infinite regular cover of $\Sigma_{g,n}$. Then $\Sigma_{g,n}^H$ is homeomorphic to one of the 11 topological types given at the beginning of this section. 
\end{theorem}

\begin{remark}
An analogous theorem is valid for the nonorientable surfaces. Just it increases the list with the corresponding six open nonorientable surfaces.
\end{remark}

Now, we give some applications of these statements on foliation theory.
Consider a compact $\mathcal{C}^r$ differential manifold $F$, with $r\geq1$, and the group representation
 \[\Phi:\pi_1(\Sigma_{g,n})\rightarrow\mathrm{Diff}^r(F),\] 
where $\mathrm{Diff}^r(F)$ is the group of $C^r$ diffeomorphisms of $F$. Let $\Sigma^{id}_{g,n}$ the universal cover of $\Sigma_{g,n}$ and $M$ the quotient space of $\Sigma^{id}_{g,n}\times F/\sim$, where the equivalence relation is $(x,y)\sim(h_\gamma(x),\Phi(\gamma)(y))$ and $h_\gamma\in A^{id}_{g,n}$. $M$ is the suspension of $\Phi$ and it comes with a projection $P_\Phi:M\rightarrow\Sigma_{g,n}$. The foliation with leaves $\Sigma^{id}_{g,n}\times\{y\}$ on $\Sigma^{id}_{g,n}\times F$ gives a new foliation $\mathcal{F}_{\Phi}$ on $M$, which leaves are covering spaces of $\Sigma_{g,n}$ and they are transversal at each point of each fiber $P_{\Phi}^{-1}(o)=F$, with $o\in\Sigma_{g,n}$ (\cite[Chapter V]{camacho1979teoria}, \cite[Section 3.1,Vol.I]{candel2003foliations}). Thus, we have the following statement:

\begin{theorem}
Let $F,\Phi,P_{\Phi}:M\rightarrow\Sigma_{g,n}\text{and}\quad\mathcal{F}_{\Phi}$ be as above. Assume the image of each canonical generator $\gamma\in\mathfrak{G}_{g,n}$ under the representation $\Phi$ is a diffeomorphism of $F$ whose set of fixed points is finite. Then, any leaf of $\mathcal{F}_{\Phi}$ outside of a countable set of leaves is homeomorphic to one of the 11 topological types given at the Theorem \ref{THM:covers}.
\end{theorem}    

\begin{proof}
Fix the fiber $F=P_{\Phi}^{-1}(o)$, with $o\in\Sigma_{g,n}$. Let $G=\Phi(\pi_1(\Sigma_{g,n}))$ be the image of the representation. The group G is known as the global holonomy of $\mathcal{F}_\Phi$. Take $p \in F$ and consider the leaf $\mathcal L_p$ of $\mathcal{F}_\Phi$ through $p$. It is a covering space of $\Sigma_{g,n}$, whose covering map is the restriction of $P_{\Phi}$ to $\mathcal{L}_p$. Note that the isotropy group $\mathrm{Iso}_G(p) = \{ g \in G | g(p) = p \}$ is a subgroup of $P_{\Phi\ast}(\pi_1(\mathcal{L}_p))$.

Since $G$ is countable and any nontrivial element in $G$ has a finite set of fixed points, it follows that
there exists a countable set $\mathcal{C} \subset F$ such that for every $p \in F - \mathcal{C}$ the group $\mathrm{Iso}_G(p)$ reduces to the identity.
Thus any two leaves $\mathcal L_p$ and $\mathcal L_q$ with $p,q \in F- \mathcal{C}$ are regular covers of $\Sigma_{g,k}$ with isomorphic groups of covering transformations. It follows that $\mathcal L_p$ and $\mathcal L_q$ are
homeomorphic.
Since $P_{\Phi\ast}(\pi_1(\mathcal{L}_p))=P_{\Phi\ast}(\pi_1(\mathcal{L}_q))=\mathrm{ker}\Phi$, Theorem \ref{THM:covers} shows that these leaves are homeomorphic to one of the real surfaces of the list in this theorem.

\end{proof}

We call \textit{Riccati} an holomorphic foliation $\mathcal{F}$ of complex dimension 1 on a compact connected complex surface if there exists a rational fibration $P:M\rightarrow\Sigma_g$, whose generic fiber is transversal to $\mathcal{F}$ and biholomorphic to $\mathbb{CP}^1$. Possibly the fibration has singular fibers, but these belong to a finite set $\{P^{-1}(p_j)\}_{j=1}^n$ of $\mathcal{F}$-invariant fibers, $\Sigma_g-\{p_j\}_{j=1}^{n}=\Sigma_{g,n}$. The restriction $\mathcal{F}|_{M*}$, with $M*=M-\cup_{j=1}^{n}P^{-1}(p_j)$, is a foliation obtained by suspending the global holonomy $\Phi_{\mathcal{F}}$ of $\mathcal{F}$
\[
\Phi_{\mathcal{F}}:\pi_1(\Sigma_{g,n})\rightarrow\mathrm{PSL}(2,\mathbb{C}),
\]
(see for instance \cite{gomez1989holomorphic}). As any biholomorphism in $\Phi_{\mathcal{F}}(\pi_1(\Sigma_{g,n}))$ has at most two fixed points, the Theorem\ref{THM:covers} implies the following result:

\begin{corollary}\label{crl:Riccati}
Let $\mathcal{F}$ be a Riccati foliation on compact	 complex surface $M$. Then any leaf of $\mathcal{F}$ outside of a countable set of leaves is homeomorphic to one of the 11 topological types given at the Theorem \ref{THM:covers}.
\end{corollary}

\begin{example}\label{exp:BCT}
For each $x\in\C$, let $\Gamma_x$ be the subgroup of $\mathrm{PSL}(2,\C)$ generated by
\[
e_1=\begin{pmatrix}
1 & x \\
0 & 1
\end{pmatrix},\quad\quad e_2=\begin{pmatrix}
1 & 0 \\
x & 1
\end{pmatrix}.\]
If $|x|\geq2$, the ping-pong lemma (\cite[Theorem 4.3.1]{loh2011geometric}) implies  that the subgroup $\Gamma_x$ is free of rank two.

 Consider the suspension of the homomorphism
\[\begin{array}{llll}
\Phi:&\pi_1(\Sigma_3,o)&\rightarrow &  \mathrm{PSL}(2,\C)\\
  & a_1 & \mapsto & e_1\\
  & a_2 & \mapsto & e_2 \\
  & a_3,b_j & \mapsto & \mathit{id},
\end{array}\]
where $\{a_j,b_j\}_{j=1}^3=\mathfrak{G}_{3,0}$. This suspension gives a nonsingular Riccati foliation $\mathcal{F}$ on a compact complex surface $M$ with adapted fibration $P_\Phi:M\rightarrow \Sigma_3$.
The generic leaf $\mathcal{L}$ of $\mathcal{F}$ is a regular cover of $\Sigma_{3}$ with its group of deck transformation isomorphic to $\Gamma_x$. In particular, $\Sigma^{H}_3\cong\mathcal{L}$ whit $H$ being the kernel of $\Phi$. From \cite[Theorem8.2.14]{loh2011geometric}, we conclude that $\mathcal{E}(\Gamma_x)$ is a infinite set. Thus, $\mathcal{E}(\mathcal{L})$ is a Cantor set. The figure below shows the closure of the union of the lifts $\tilde{D}_{3}(\tilde{o}')$ of the fundamental domain $D_3$ of $\Sigma_3$ at each vertex in the ball $B_4(\tilde{o})\subset\mathcal{L}$. This closed surface has nontrivial genus, then $\mathcal{E}'(\mathcal{L})$ is a Cantor set. It follows that $\mathcal{L}$ is homeomorphic to the blooming Cantor tree.

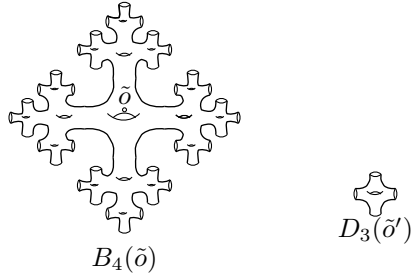
\begin{tikzpicture}[scale=.75]
%\draw[step=1mm,green,very thin] (5.8,-1) grid (11,4.5);
%\draw[step=5mm,red,very thin] (5.8,-1) grid (11,4);
%\draw (9.4,-1.3) node {Figure 8};
\foreach \x in {2,13}
\foreach \y in {-1,4.5}
{
\draw[white] (\x,\y) circle (1pt);
}

\draw[black] (8,2.14) circle (1pt);
\draw (8,2.4) node {$\tilde{o}$};
%Cantor Tree

%First nivel
\foreach \x in {8.2}
\foreach \y in {2.7}
{
\draw[line width=0.5pt] (\x,\y) .. controls (\x-.02,\y-.2) .. (\x+.1,\y-.4) .. controls (\x+.2,\y-.5) and (\x+.4,\y-.52) .. (\x+.5,\y-.5);
}
\foreach \x in {7.8}
\foreach \y in {2.7}
{
\draw[line width=0.5pt] (\x,\y) .. controls (\x+.02,\y-.2) .. (\x-.1,\y-.4) .. controls (\x-.2,\y-.5) and (\x-.4,\y-.52) .. (\x-.5,\y-.5);
}
\foreach \x in {8.2}
\foreach \y in {1.3}
{
\draw[line width=0.5pt] (\x,\y) .. controls (\x-.02,\y+.2) .. (\x+.1,\y+.4) .. controls (\x+.2,\y+.5) and (\x+.4,\y+.52) .. (\x+.5,\y+.5);
}
\foreach \x in {7.8}
\foreach \y in {1.3}
{
\draw[line width=0.5pt] (\x,\y) .. controls (\x+.02,\y+.2) .. (\x-.1,\y+.4) .. controls (\x-.2,\y+.5) and (\x-.4,\y+.52) .. (\x-.5,\y+.5);
}
%olho
\draw (7.8,2) arc (135:45:.28);
\draw (7.7,2.075) arc (225:315:.4);
%Second nivel
%olhos
\draw (8.95,2) arc (135:45:.14);
\draw (8.9,2.05) arc (225:315:.2);
\draw (6.85,2) arc (135:45:.14);
\draw (6.8,2.05) arc (225:315:.2);
\draw (7.9,3.1) arc (135:45:.14);
\draw (7.85,3.15) arc (225:315:.2);
\draw (7.9,0.9) arc (135:45:.14);
\draw (7.85,0.95) arc (225:315:.2);
%second nivel esquinas interiores
\foreach \x in {8.4}
\foreach \y in {2.9}
{
\draw[line width=0.5pt] (\x,\y) .. controls (\x,\y+.1) .. (\x-.1,\y+.13) .. controls (\x-.2,\y+.1) and (\x-.22,\y-.02) .. (\x-.2,\y-.2);
}

\foreach \x in {7.6}
\foreach \y in {1.1}
{
\draw[line width=0.5pt] (\x,\y) .. controls (\x,\y-.1) .. (\x+.1,\y-.13) .. controls (\x+.2,\y-.1) and (\x+.22,\y+.02) .. (\x+.2,\y+.2);
}

\foreach \x in {7.6}
\foreach \y in {2.9}
{
\draw[line width=0.5pt] (\x,\y) .. controls (\x,\y+.1) .. (\x+.1,\y+.13) .. controls (\x+.2,\y+.1) and (\x+.22,\y-.02) .. (\x+.2,\y-.2);
}

\foreach \x in {8.4}
\foreach \y in {1.1}
{
\draw[line width=0.5pt] (\x,\y) .. controls (\x,\y-.1) .. (\x-.1,\y-.13) .. controls (\x-.2,\y-.1) and (\x-.22,\y+.02) .. (\x-.2,\y+.2);
}

\foreach \x in {8.9}
\foreach \y in {2.4}
{
\draw[line width=0.5pt] (\x,\y) .. controls (\x+.1,\y) .. (\x+.13,\y-.1) .. controls (\x+.1,\y-.2) and (\x-.02,\y-.22) .. (\x-.2,\y-.2);
}
\foreach \x in {8.9}
\foreach \y in {1.6}
{
\draw[line width=0.5pt] (\x,\y) .. controls (\x+.1,\y) .. (\x+.13,\y+.1) .. controls (\x+.1,\y+.2) and (\x-.02,\y+.22) .. (\x-.2,\y+.2);
}
\foreach \x in {7.1}
\foreach \y in {1.6}
{
\draw[line width=0.5pt] (\x,\y) .. controls (\x-.1,\y) .. (\x-.13,\y+.1) .. controls (\x-.1,\y+.2) and (\x+.02,\y+.22) .. (\x+.2,\y+.2);
}
\foreach \x in {7.1}
\foreach \y in {2.4}
{
\draw[line width=0.5pt] (\x,\y) .. controls (\x-.1,\y) .. (\x-.13,\y-.1) .. controls (\x-.1,\y-.2) and (\x+.02,\y-.22) .. (\x+.2,\y-.2);
}
%second nivel esquinas exeriores
\foreach \x in {6.5}
\foreach \y in {2.4}
{
\draw[line width=0.5pt] (\x,\y) .. controls (\x+.13,\y-.02) and (\x+.2,\y-.13) .. (\x+.15,\y-.26) .. controls (\x+.03,\y-.3) and (\x-.12,\y-.33) .. (\x-.1,\y-.1);
}
\foreach \x in {7.7}
\foreach \y in {3.6}
{
\draw[line width=0.5pt] (\x,\y) .. controls (\x+.13,\y-.02) and (\x+.2,\y-.13) .. (\x+.15,\y-.26) .. controls (\x+.03,\y-.3) and (\x-.12,\y-.33) .. (\x-.1,\y-.1);
}
\foreach \x in {8.3}
\foreach \y in {3.6}
{
\draw[line width=0.5pt] (\x,\y) .. controls (\x-.13,\y-.02) and (\x-.2,\y-.13) .. (\x-.15,\y-.26) .. controls (\x-.03,\y-.3) and (\x+.12,\y-.33) .. (\x+.1,\y-.1);
}
\foreach \x in {9.5}
\foreach \y in {2.4}
{
\draw[line width=0.5pt] (\x,\y) .. controls (\x-.13,\y-.02) and (\x-.2,\y-.13) .. (\x-.15,\y-.26) .. controls (\x-.03,\y-.3) and (\x+.12,\y-.33) .. (\x+.1,\y-.1);
}
\foreach \x in {9.5}
\foreach \y in {1.6}
{
\draw[line width=0.5pt] (\x,\y) .. controls (\x-.13,\y+.02) and (\x-.2,\y+.13) .. (\x-.15,\y+.26) .. controls (\x-.03,\y+.3) and (\x+.12,\y+.33) .. (\x+.1,\y+.1);
}
\foreach \x in {8.3}
\foreach \y in {0.4}
{
\draw[line width=0.5pt] (\x,\y) .. controls (\x-.13,\y+.02) and (\x-.2,\y+.13) .. (\x-.15,\y+.26) .. controls (\x-.03,\y+.3) and (\x+.12,\y+.33) .. (\x+.1,\y+.1);
}
\foreach \x in {7.7}
\foreach \y in {0.4}
{
\draw[line width=0.5pt] (\x,\y) .. controls (\x+.13,\y+.02) and (\x+.2,\y+.13) .. (\x+.15,\y+.26) .. controls (\x+.03,\y+.3) and (\x-.12,\y+.33) .. (\x-.1,\y+.1);
}
\foreach \x in {6.5}
\foreach \y in {1.6}
{
\draw[line width=0.5pt] (\x,\y) .. controls (\x+.13,\y+.02) and (\x+.2,\y+.13) .. (\x+.15,\y+.26) .. controls (\x+.03,\y+.3) and (\x-.12,\y+.33) .. (\x-.1,\y+.1);
}
%third nevel
\draw (8.95,2) arc (135:45:.14);
\draw (8.9,2.05) arc (225:315:.2);
%olhos
\draw (9.625,2) arc (135:45:.07);
\draw (9.6,2.025) arc (225:315:.1);
\draw (9.125,1.5) arc (135:45:.07);
\draw (9.1,1.525) arc (225:315:.1);
\draw (9.125,2.5) arc (135:45:.07);
\draw (9.1,2.525) arc (225:315:.1);
\draw (8.425,3.2) arc (135:45:.07);
\draw (8.4,3.225) arc (225:315:.1);
\draw (7.925,3.7) arc (135:45:.07);
\draw (7.9,3.725) arc (225:315:.1);
\draw (7.425,3.2) arc (135:45:.07);
\draw (7.4,3.225) arc (225:315:.1);
\draw (6.725,2.5) arc (135:45:.07);
\draw (6.7,2.525) arc (225:315:.1);
\draw (6.225,2) arc (135:45:.07);
\draw (6.2,2.025) arc (225:315:.1);
\draw (6.725,1.5) arc (135:45:.07);
\draw (6.7,1.525) arc (225:315:.1);
\draw (7.425,0.8) arc (135:45:.07);
\draw (7.4,0.825) arc (225:315:.1);
\draw (7.925,0.3) arc (135:45:.07);
\draw (7.9,0.325) arc (225:315:.1);
\draw (8.425,0.8) arc (135:45:.07);
\draw (8.4,0.825) arc (225:315:.1);
%esquina derecha arriba 3er nivel
\foreach \x in {7.1}
\foreach \y in {2.6}
{
\draw[line width=0.5pt] (\x,\y) .. controls (\x-.1,\y-.04) .. (\x-.2,\y-.03) .. controls (\x-.23,\y) and (\x-.21,\y+.1) .. (\x-.2,\y+.2);
}
%esquina izquierda arriba 3er nivel
\foreach \x in {6.5}
\foreach \y in {2.6}
{
\draw[line width=0.5pt] (\x,\y) .. controls (\x+.1,\y-.04) .. (\x+.2,\y-.03) .. controls (\x+.23,\y) and (\x+.21,\y+.1) .. (\x+.2,\y+.2);
}

\foreach \x in {7.2}
\foreach \y in {0.9}
{
\draw[line width=0.5pt] (\x,\y) .. controls (\x+.1,\y-.04) .. (\x+.2,\y-.03) .. controls (\x+.23,\y) and (\x+.21,\y+.1) .. (\x+.2,\y+.2);
}
\foreach \x in {6}
\foreach \y in {2.1}
{
\draw[line width=0.5pt] (\x,\y) .. controls (\x+.1,\y-.04) .. (\x+.2,\y-.03) .. controls (\x+.23,\y) and (\x+.21,\y+.1) .. (\x+.2,\y+.2);
}
\foreach \x in {7.2}
\foreach \y in {3.3}
{
\draw[line width=0.5pt] (\x,\y) .. controls (\x+.1,\y-.04) .. (\x+.2,\y-.03) .. controls (\x+.23,\y) and (\x+.21,\y+.1) .. (\x+.2,\y+.2);
}
\foreach \x in {7.7}
\foreach \y in {3.8}
{
\draw[line width=0.5pt] (\x,\y) .. controls (\x+.1,\y-.04) .. (\x+.2,\y-.03) .. controls (\x+.23,\y) and (\x+.21,\y+.1) .. (\x+.2,\y+.2);
}
\foreach \x in {7.7}
\foreach \y in {3.8}
{
\draw[line width=0.5pt] (\x,\y) .. controls (\x+.1,\y-.04) .. (\x+.2,\y-.03) .. controls (\x+.23,\y) and (\x+.21,\y+.1) .. (\x+.2,\y+.2);
}
\foreach \x in {8.9}
\foreach \y in {2.6}
{
\draw[line width=0.5pt] (\x,\y) .. controls (\x+.1,\y-.04) .. (\x+.2,\y-.03) .. controls (\x+.23,\y) and (\x+.21,\y+.1) .. (\x+.2,\y+.2);
}
%esquina izquierda inferior 3er nivel
\foreach \x in {6}
\foreach \y in {1.9}
{
\draw[line width=0.5pt] (\x,\y) .. controls (\x+.1,\y+.04) .. (\x+.2,\y+.03) .. controls (\x+.23,\y) and (\x+.21,\y-.1) .. (\x+.2,\y-.2);
}
\foreach \x in {6.5}
\foreach \y in {1.4}
{
\draw[line width=0.5pt] (\x,\y) .. controls (\x+.1,\y+.04) .. (\x+.2,\y+.03) .. controls (\x+.23,\y) and (\x+.21,\y-.1) .. (\x+.2,\y-.2);
}
\foreach \x in {7.2}
\foreach \y in {.7}
{
\draw[line width=0.5pt] (\x,\y) .. controls (\x+.1,\y+.04) .. (\x+.2,\y+.03) .. controls (\x+.23,\y) and (\x+.21,\y-.1) .. (\x+.2,\y-.2);
}
\foreach \x in {7.7}
\foreach \y in {.2}
{
\draw[line width=0.5pt] (\x,\y) .. controls (\x+.1,\y+.04) .. (\x+.2,\y+.03) .. controls (\x+.23,\y) and (\x+.21,\y-.1) .. (\x+.2,\y-.2);
}
\foreach \x in {8.9}
\foreach \y in {1.4}
{
\draw[line width=0.5pt] (\x,\y) .. controls (\x+.1,\y+.04) .. (\x+.2,\y+.03) .. controls (\x+.23,\y) and (\x+.21,\y-.1) .. (\x+.2,\y-.2);
}
\foreach \x in {7.2}
\foreach \y in {3.1}
{
\draw[line width=0.5pt] (\x,\y) .. controls (\x+.1,\y+.04) .. (\x+.2,\y+.03) .. controls (\x+.23,\y) and (\x+.21,\y-.1) .. (\x+.2,\y-.2);
}
%esquina derecha arriba 3er nivel
\foreach \x in {8.3}
\foreach \y in {3.8}
{
\draw[line width=0.5pt] (\x,\y) .. controls (\x-.1,\y-.04) .. (\x-.2,\y-.03) .. controls (\x-.23,\y) and (\x-.21,\y+.1) .. (\x-.2,\y+.2);
}
\foreach \x in {8.8}
\foreach \y in {3.3}
{
\draw[line width=0.5pt] (\x,\y) .. controls (\x-.1,\y-.04) .. (\x-.2,\y-.03) .. controls (\x-.23,\y) and (\x-.21,\y+.1) .. (\x-.2,\y+.2);
}
\foreach \x in {9.5}
\foreach \y in {2.6}
{
\draw[line width=0.5pt] (\x,\y) .. controls (\x-.1,\y-.04) .. (\x-.2,\y-.03) .. controls (\x-.23,\y) and (\x-.21,\y+.1) .. (\x-.2,\y+.2);
}
\foreach \x in {10}
\foreach \y in {2.1}
{
\draw[line width=0.5pt] (\x,\y) .. controls (\x-.1,\y-.04) .. (\x-.2,\y-.03) .. controls (\x-.23,\y) and (\x-.21,\y+.1) .. (\x-.2,\y+.2);
}
\foreach \x in {8.8}
\foreach \y in {0.9}
{
\draw[line width=0.5pt] (\x,\y) .. controls (\x-.1,\y-.04) .. (\x-.2,\y-.03) .. controls (\x-.23,\y) and (\x-.21,\y+.1) .. (\x-.2,\y+.2);
}
%esquina inferior derecha
\foreach \x in {8.8}
\foreach \y in {3.1}
{
\draw[line width=0.5pt] (\x,\y) .. controls (\x-.1,\y+.04) .. (\x-.2,\y+.03) .. controls (\x-.23,\y) and (\x-.21,\y-.1) .. (\x-.2,\y-.2);
}
\foreach \x in {10}
\foreach \y in {1.9}
{
\draw[line width=0.5pt] (\x,\y) .. controls (\x-.1,\y+.04) .. (\x-.2,\y+.03) .. controls (\x-.23,\y) and (\x-.21,\y-.1) .. (\x-.2,\y-.2);
}
\foreach \x in {9.5}
\foreach \y in {1.4}
{
\draw[line width=0.5pt] (\x,\y) .. controls (\x-.1,\y+.04) .. (\x-.2,\y+.03) .. controls (\x-.23,\y) and (\x-.21,\y-.1) .. (\x-.2,\y-.2);
}
\foreach \x in {8.8}
\foreach \y in {.7}
{
\draw[line width=0.5pt] (\x,\y) .. controls (\x-.1,\y+.04) .. (\x-.2,\y+.03) .. controls (\x-.23,\y) and (\x-.21,\y-.1) .. (\x-.2,\y-.2);
}
\foreach \x in {8.3}
\foreach \y in {.2}
{
\draw[line width=0.5pt] (\x,\y) .. controls (\x-.1,\y+.04) .. (\x-.2,\y+.03) .. controls (\x-.23,\y) and (\x-.21,\y-.1) .. (\x-.2,\y-.2);
}
\foreach \x in {7.1}
\foreach \y in {1.4}
{
\draw[line width=0.5pt] (\x,\y) .. controls (\x-.1,\y+.04) .. (\x-.2,\y+.03) .. controls (\x-.23,\y) and (\x-.21,\y-.1) .. (\x-.2,\y-.2);
}
%simulate surface pieces
%cachete pa abajo
\draw (7.9,4) arc (225:315:.14);
\draw (8.4,3.5) arc (225:315:.14);
\draw (9.1,2.8) arc (225:315:.14);
\draw (9.6,2.3) arc (225:315:.14);
\draw (7.4,3.5) arc (225:315:.14);
\draw (6.7,2.8) arc (225:315:.14);
\draw (6.2,2.3) arc (225:315:.14);
\draw (8.4,1.1) arc (225:315:.14);
\draw (7.4,1.1) arc (225:315:.14);
\draw (8.4,.5) arc (225:315:.14);
\draw (7.4,.5) arc (225:315:.14);
\draw (8.4,2.9) arc (225:315:.14);
\draw (7.4,2.9) arc (225:315:.14);
\draw (9.1,1.2) arc (225:315:.14);
\draw (9.6,1.7) arc (225:315:.14);
\draw (7.9,0) arc (225:315:.14);
\draw (6.7,1.2) arc (225:315:.14);
\draw (6.2,1.7) arc (225:315:.14);
%cachete pa arriba
\draw (7.9,4) arc (135:45:.14);
\draw (8.4,3.5) arc (135:45:.14);
\draw (9.1,2.8) arc (135:45:.14);
\draw (9.6,2.3) arc (135:45:.14);
\draw (7.4,3.5) arc (135:45:.14);
\draw (6.7,2.8) arc (135:45:.14);
\draw (6.2,2.3) arc (135:45:.14);
\draw (8.4,1.1) arc (135:45:.14);
\draw (7.4,1.1) arc (135:45:.14);
%cachetes izquierdos
\draw (7.7,3.8) arc (30:-30:.195);
\draw (8.3,3.8) arc (30:-30:.195);
\draw (7.2,3.3) arc (30:-30:.195);
\draw (7.1,2.6) arc (30:-30:.195);
\draw (6.5,2.6) arc (30:-30:.195);
\draw (6,2.1) arc (30:-30:.195);
\draw (6.5,1.6) arc (30:-30:.195);
\draw (7.1,1.6) arc (30:-30:.195);
\draw (7.2,.9) arc (30:-30:.195);
\draw (7.7,.4) arc (30:-30:.195);
\draw (8.3,.4) arc (30:-30:.195);
\draw (8.8,.9) arc (30:-30:.195);
\draw (8.9,1.6) arc (30:-30:.195);
\draw (9.5,1.6) arc (30:-30:.195);
\draw (10,2.1) arc (30:-30:.195);
\draw (9.5,2.6) arc (30:-30:.195);
\draw (8.9,2.6) arc (30:-30:.195);
\draw (8.8,3.3) arc (30:-30:.195);
%cachete derecho
\draw (7.7,3.8) arc (150:210:.195);
\draw (8.3,3.8) arc (150:210:.195);
\draw (7.2,3.3) arc (150:210:.195);
\draw (7.1,2.6) arc (150:210:.195);
\draw (6.5,2.6) arc (150:210:.195);
\draw (6,2.1) arc (150:210:.195);
\draw (6.5,1.6) arc (150:210:.195);
\draw (7.1,1.6) arc (150:210:.195);
\draw (7.2,.9) arc (150:210:.195);
\draw (7.7,.4) arc (150:210:.195);
\draw (8.3,.4) arc (150:210:.195);
\draw (8.8,.9) arc (150:210:.195);
\draw (9.5,1.6) arc (150:210:.195);
\draw (10,2.1) arc (150:210:.195);
\draw (9.5,2.6) arc (150:210:.195);
\draw (8.9,2.6) arc (150:210:.195);
\draw (8.8,3.3) arc (150:210:.195);
\draw (8.9,1.6) arc (150:210:.195);

\draw (8,-0.5) node {$B_4(\tilde{o})$};
%pieza fundamental levantada
\foreach \x in {12.5}
\foreach \y in {1}
{
\draw[line width=0.5pt] (\x,\y) .. controls (\x-.01,\y-.1) .. (\x+.05,\y-.2) .. controls (\x+.1,\y-.25) and (\x+.2,\y-.26) .. (\x+.25,\y-.25);
}
\foreach \x in {12.3}
\foreach \y in {1}
{
\draw[line width=0.5pt] (\x,\y) .. controls (\x+.01,\y-.1) .. (\x-.05,\y-.2) .. controls (\x-.1,\y-.25) and (\x-.2,\y-.26) .. (\x-.25,\y-.25);
}
\foreach \x in {12.5}
\foreach \y in {.3}
{
\draw[line width=0.5pt] (\x,\y) .. controls (\x-.01,\y+.1) .. (\x+.05,\y+.2) .. controls (\x+.1,\y+.25) and (\x+.2,\y+.26) .. (\x+.25,\y+.25);
}
\foreach \x in {12.3}
\foreach \y in {.3}
{
\draw[line width=0.5pt] (\x,\y) .. controls (\x+.01,\y+.1) .. (\x-.05,\y+.2) .. controls (\x-.1,\y+.25) and (\x-.2,\y+.26) .. (\x-.25,\y+.25);
}
\draw (12.3,.3) arc (225:315:.14);
\draw (12.3,1) arc (225:315:.14);
\draw (12.3,1) arc (135:45:.14);
\draw (12.05,.75) arc (150:210:.195);
\draw (12.75,.75) arc (150:210:.195);
\draw (12.05,.75) arc (30:-30:.195);
\draw (12.75,.75) arc (30:-30:.195);
\draw (12.4,0) node {$D_3(\tilde{o}')$};
\draw (12.3,.65) arc (135:45:.14);
\draw (12.25,.7) arc (225:315:.2);

\end{tikzpicture}

The blooming Cantor tree without an infinite discrete set as the generic leaf is obtained from a bimeromorphism $\varphi:M\rightarrow M'$, which is the composition of one blow-up $b$ in a regular point $p$ in a fibre $F$ and one blow-down on the closure of $b^{-1}(F)-E$, where $E=b^{-1}(p)$ (see \cite[p.54]{brunella2004}). The bimeromorphism sends a trivial neighborhood $U\subset M$ of a fibre $F$ to a trivial neighborhood $U'\subset M'$ with an induced foliation, which has two singularities on the fibre $F'$, one logarithmic and one dicritical. The holonomy around $F'$ is trivial. Therefore the generic leaf $\mathcal{L}$ is a normal cover of $\Sigma_{3,1}$. Considering that the holonomy of the border cycle is trivial, we have that $\mathcal{L}$ is homeomorphic to the blooming Cantor tree deprived from a infinite discrete set.
\end{example}

\begin{example}\label{exp:CT}
Let $\Gamma_x,e_1,e_2$ be as in example above, with $x\geq2$. Consider the foliation $\mathcal{F}$ on the suspension $P_\Phi:M\rightarrow\Sigma_{2}$ generated by the representation 
\[\begin{array}{llll}
\Phi:&\pi_1(\Sigma_2,o)&\rightarrow &  \mathrm{PSL}(2,\C)\\
  & a_1 & \mapsto & e_1\\
  & a_2 & \mapsto & e_2 \\
  & b_j & \mapsto & \mathit{id},
\end{array}\]
where $\{a_j,b_j\}_{j=1}^2=\mathfrak{G}_{2,0}$. An analysis similar to that in the developing of Example \ref{exp:BCT} shows that the generic leaf of $\mathcal{F}$ is homeomorphic to the Cantor tree. Also using the composition of a blow-up at point $p\in F$ and a blow-down of the strict transform of $F$, we get a bimeromorphism $\phi:M\rightarrow M'$. The foliation $\phi^{-1\ast}\mathcal{F}$ on $M'$ has generic leaf homeomorphic to Cantor tree without an infinite discrete set.
\end{example}

\begin{example}
Let $\Sigma_2$ be the compact Riemann surface of genus 2 and $\{a_1,b_1,a_2,b_2\}$ be the canonical generators of $\pi_1(\Sigma_2,o)$ for a point $o$ in $\Sigma_2$. Consider the homomorphism defined by
\[\begin{array}{lll}
\rho:\pi_1(\Sigma_2,o)&\rightarrow & \mathrm{PSL}(2,\C)\\
a_1 &\mapsto & T\\
a_2,b_j &\mapsto &\mathit{id},
\end{array}\]
where $T(z)=z+1$, which fixed a sinlge point at infinity. From the suspension of this homomorphism we get a non singular Riccati foliation with a compact leaf.
Since $\rho(\pi_1(\Sigma_2,o))$ is isomorphic to $\mathbb{Z}$, \cite[Theorem 8.2.14]{loh2011geometric} shows that any generic leaf $\mathcal{L}$ of $\mathcal{F}$ has two ends.  We conclude from the construct of the surface $B_n(\tilde{o})$, with $n\geq1$,  that the generic leaf $\mathcal{L}$ has infinite genus. Therefore $\mathcal{L}$ is homeomorphic to Jacob's ladder. Using a bimeromorphism as in Examples \ref{exp:BCT} and \ref{exp:CT} , it is possible to give an example of foliation with generic leaf homeomorphic to Jacob's ladder without an infinite discrete set.\end{example}

We now consider holomorphic homogeneous foliations $\mathcal{F}$ in $\mathbb{CP}^2$
which are defined in an affine chart $U\subset\mathbb{CP}^2$ by a homogeneous $1$-form
\[
    \omega=h_{1}(x,y)dx+h_2(x,y)dy\quad\text{and}\quad\omega(R)\neq0
\]
where $h_1,h_2$ are homogeneous polynomials of the same degree $\nu$, without common factors and $R$ is the radial vector field. The foliation $\mathcal{F}$ defined by $\omega$ extends to a foliation of $\mathbb{CP}^2$ leaving the line at infinity invariant.
After making a linear change of coordinates we can  assume that
$x$ does not divide $\omega(R) = x h_1 + y h_2$.
The roots of this polynomial $\omega(R)=\prod_{1}^{n}(y-t_{j}x)^{\nu_j}$
correspond to $\mathcal{F}$-invariant lines $l_j$ through the origin. 

Blowing-up $(0,0) \in U$ we obtain a Riccati foliation $\mathcal{F}'$ 
with a finite set $\{\tilde{l}_j\}_{j=1}^n$ of $\mathcal{F}'$-invariants fibers corresponding to the strict transforms of the lines $l_j$, and fibration $P:M\rightarrow\Sigma_0$. The global holonomy group $\Phi_{\mathcal{F}'}(\pi_1(\Sigma_{0,n}))$ of $\mathcal{F}'$ is linear and isomorphic to the group generated by $<\exp^{2\pi i\eta_j}z>$, where $\eta_j$ is the Camacho-Sad index of the singularity of $\mathcal{F}'$ at $p_j=\tilde{l}_j\cap E$ in the exceptional divisor $E$ \cite{brunella2004}. In particular, the global holonomy $\Phi_{\mathcal{F}'}(\pi_1(\Sigma_{0,n}))$ is abelian and isomorphic to the multiplicative group $<\exp(2\pi i\eta_j)>_{j=1}^n$ in $\mathbb C^*$. Except for the leaves corresponding to invariant fibers, the exceptional divisor $E$ and the line at infinity, all the other leaves $\mathcal{L}$ are homeomorphic to a regular cover of $\Sigma_{0,n}$ with the same group of deck transformations. We will describe the possible topological types of this leaves. 

In the proofs of \cite[Theorem 1]{valdez2009billiards} and \cite[Theorem 8]{goncharuk2014genera}, the authors show that for generic $\{\eta_j\}$ and $n\geq3$ it is possible to give two cycles $\gamma_1,\gamma_2$ in $[\pi_1(\Sigma_{0,n},o),\pi_1(\Sigma_{0,n},o)]$ whose lifts $\gamma_1',\gamma_2'$ at $p\in\mathcal{L}\cap F$ are nontrivial in $\pi_1(L)$ and have intersection index equal to 1, with $F=P^{-1}(o)$. Thus the leaf has an attached handle. But this hold for any point in $\mathcal{L}\cap F$ since $\mathcal{L}$ has infinite genus. The following statement is an extension of the construction of these two cycles with intersection index 1 in regular covers $\Sigma_{g,n}^H$, with $[\pi_1(\Sigma_{g,n}),\pi_1(\Sigma_{g,n})]\cap H\neq\{\mathit{id}\}$.

\begin{proposition}\label{Propt:AbelianCovers}
Let $\Sigma_{g,n},\Sigma_{g,n}^H,H$ and $\mathfrak{G}_{g,n}$ be as defined above. The following assertions give conditions on $g,n$ and $\mathfrak{G}_{g,n}$ to have to cycles with intersection index one.
\begin{itemize}
\item[a)] For $g=0,n\geq3$ and some pair $c_j,c_l\in\mathfrak{G}_{0,n}$. If it is satisfies that the cycles $\gamma_1=[c_i^m,c_j^l],\gamma_2=[c_i^{-m'},c_j^l]$ belong to $H$ and $c_i^\alpha,c_j^\beta\notin H$ for $\alpha\leq m+m'$ and $\beta\leq l$, with $m,m',l\in\mathbb{N}_{>0}$.
\item[b)] For $g\geq1,n\geq0$ and $a_j,b_j,d\in\mathfrak{G}_{g,n}$, with $a_j\neq d\neq b_j$ and $2g+n\geq 3$. 
If it is satisfies that the cycles $\gamma_1=a_j^m d^{-k}a_j^{-2m}d^k a_j^m$ and $gamma_2=b_j^l d^{k'}b_j^{-2l}d^{-k'} b_j^l$ are contained in $H$ and $a_j^\alpha,b_j^\beta d^\delta\notin H$ for $\alpha\leq m,\beta\leq l$, $\delta\leq k+k'$,with $m,l,k,k'\in\mathbb{N}_{>0}$. 
In the case, there are $m$ or $l$ in $\mathbb{N}_{>0}$ such that they are minimal with the property $a_j^m\in H$ or $b_j^l \in H$. 
It is possible to change the cycles $\gamma_1$ by $a_j^m$ or $\gamma_2$ by $b_j^l$.

\end{itemize}
Then the lifts $\tilde{\gamma_1},\tilde{\gamma_2}$ of the cycles $\gamma_1,\gamma_2$ at a point $\tilde{o}\in\Sigma_{g,n}^H$ have intersection index one. Moreover, if $\Sigma_{g,n}^H$ is a infinite cover then it has infinite genus.
\end{proposition}
\begin{proof}
The proof is based in the following observations. The lifts $\tilde{\gamma_1},\tilde{\gamma_2}$ of the cycles $\gamma_1,\gamma_2$ at a point $\tilde{o}\in\Sigma_{g,n}^H$ has intersection the point $\tilde{o}$. Consider the border $B$ of the lift $\tilde{D}_{g,n}(\delta,\tilde{o})$ of $D_{g,n}(\delta)$, which is a simple closed curve. The complement $B-\tilde{\gamma_1}$ has two connected components, and $\tilde{\gamma_2}$ intersects both components in only one point. The latest implies that the intersection index of $\tilde{\gamma_1}$ and $\tilde{\gamma_2}$ is one. A detailed proof of the last claim is in \cite[Section 2.4.2]{PincheTesis}.
\end{proof}
This result lets us give the topological type of the generic leaf homogeneous foliation using the data $\{\eta_j\}$ and the polynomial $\omega(R)$.   Note that the generic leaf $\mathcal{L}$ of any homogeneous foliation is homeomorphic to a regular cover $\Sigma_{0,n}^H$ such that $[\pi_1(\Sigma_{0,n}),\pi_1(\Sigma_{0,n})]$ is contained in $H$.
If $\Sigma_{0,n}$ is an infinite cover then Lemma \ref{lm:planarends} shows that $\mathcal{E}(\mathcal{L})$ is an infinite discrete set of points or a single point and Lemma \ref{lm:ends'} implies that $\mathcal{E}'(\mathcal{L})$ is an empty set or a single point, which proves the following result.
\begin{corollary}
Let $\mathcal{F}$ a homogeneous foliation in $\mathbb{CP}^2$. If the global holonomy is infinite, then the generic leaf of $\mathcal{F}$ is homeomorphic to one of the following n topological types:a plane, a cylinder, a plane with an infinite discrete set taken out, the Loch Ness monster or the Loch Ness monster without an infinite discrete set. 
\end{corollary}

\begin{example}\label{ex:plane}
Let $\mathcal{F}$ be a homogeneous foliation defined by a homogeneous 1-form
\[
\omega=ydx+\lambda xdy.
\]
If $\lambda\notin \mathbb{Q}$, Corollary \ref{crl:Riccati} implies that the generic leaves $\mathcal{L}$ of $\mathcal{F}$ are biholomorphic to an infinite normal cover of $\C^\ast$. Moreover, $\mathcal{L}$ is the universal cover of $\C^\ast$. Then it is biholomorphic to $\C$. When $\lambda$ is a rational number and $\omega(R)\neq0$ the generic leaf is biholomorphic to $\C^{\ast}$.
\end{example}

\begin{example}\label{exp:logFol1}
Consider a homogeneous foliation $\mathcal{F}$ on $\mathbb{CP}^2$ defined by the homogeous 1-form \[
\omega=\lambda_1\frac{dx}{x}+\lambda_2\frac{dy}{y}+\lambda_3\frac{d(y-x)}{(y-x)},\] where $\sum\lambda_j=1$.

Requiring for the set $\{\lambda_1,\lambda_2,\lambda_3\}$ to satisfy: \[\lambda_1\in \mathbb{R}-\mathbb{Q},\quad\lambda_2=1\quad\text{and}\quad\lambda_3=-\lambda_1,\] we obtain that the generic leaf $\mathcal{L}$ is homeomorphic to a regular cover $\Sigma_{0,3}^H$ with deck transformation group isomorphic to $(\mathbb{Z}^2)/(0,1)$. Thus $H$ does not contain the cycles of the Proposition \ref{Propt:AbelianCovers}, so that it is a planar surface. Lemma \ref{lm:planarends} now shows that $\mathrm{Ends}(\mathcal{L})$ contains an infinite discrete set of planar ends. Hence $\mathcal{L}$ is homeomorphic to the plane without an infinite discrete set.
Assuming the data is
\[\lambda_1\in \mathbb{R}-\mathbb{Q},\quad\lambda_2=\frac{1}{n}\quad \text{with}\quad n\geq2\quad\text{and}\quad\lambda_3=1-\lambda_1-\lambda_2.
\]
It follows that $\mathcal{L}$ has desk transformation group isomorphic to $(\mathbb{Z}^2)/(0,n)$. Since $n\geq2$, we see that $\mathrm{Cayley}(\mathbb{Z}^2)/(0,n)$ contains the cycles of Proposition \ref{Propt:AbelianCovers}. By Lemma \ref{lm:planarends}, $\mathrm{Ends}(\mathcal{L})$ contains an infinite discrete set of planar ends. Consequently, $\mathcal{L}$ is homeomorphic to the Loch Ness monster without an infinite discrete set. Now if the data satisfies
\[\lambda_1,\lambda_2\in \mathbb{C}-\mathbb{Q},\quad\text{and}\quad\lambda_3=1-\lambda_1-\lambda_2.
\]
Analogously, we prove that the generic leaf has infinite genus but has only one end. Hence, the generic leaf is homeomorphic to the Loch Ness monster. 
\end{example}  

\section{Reeb local stability type result}
We consider a singular $\mathcal{C}^r$ differential foliation $\mathcal{F}$ ($r\geq1$) of real dimension 2  on a differential manifold $M$ of real dimension at least 3, whose singular locus $\mathrm{Sing}(\mathcal{F})$ satisfies the following conditions:
 $\mathrm{Sing}(\mathcal{F})$ is the union of closed submanifolds, maybe some of them are singular; the complement $M^*=M-\mathrm{Sing}(\mathcal{F})$ is a dense open subset and the restriction $\mathcal{F}|_{M^*}$ is a regular foliation of real dimension 2.
 
 \begin{definition}
 Fix a point $o$ on a regular leaf $\mathcal{L}$ of $\mathcal{F}$ and a germ of transversal $\tau$ to $\mathcal{F}$ at $o$. The \textbf{holonomy group} $\mathrm{Hol}(\mathcal{L},\mathcal{F})$ is the image of the holonomy representation of $\pi_1(\mathcal{L},o)$ defined by the following morphism

\begin{equation}\label{map:holonomyRepresentation}
\begin{array}{lll}
 \mathrm{Hol}(\mathcal{L},\mathcal{F}):\pi_{1}(\mathcal{L},o)&\rightarrow&\mathrm{Diff}(\tau,o)\\
   \qquad\qquad\qquad[\gamma]&\mapsto& h_{\gamma},
 \end{array}
\end{equation}
where $h_\gamma$ is the $\mathcal{C}^r$ diffeomorphism germ of $(\tau,o)$ defined by the end point of the lifts $\tilde{\gamma}_p:[0,1]\rightarrow\mathcal{L}_p$ of the closed path $\gamma$ in $\mathcal{L}$ starting at $o$ along the leaves $\mathcal{L}_p$ through $p\in (\tau,o)$, with $\tilde{\gamma}(0)=p$. The map $h_\gamma$ does not depend on the homotopy class of the path.
\end{definition}

If $C$ is a regular compact leaf of $\mathcal{F}$ and has finite holonomy group, the Reeb local stability theorem implies that there is a neighborhood $U$ of $C$ saturated by $\mathcal{F}$ and each leaf in $U$ is a finite cover space of $C$.  Consider the case when a  compact surface $\Sigma_g$ is $\mathcal{F}$-invariant, such intersection  $\mathrm{Sing}(\mathcal{F})\cap\Sigma_g$ is a finite set of $n$ points, with $n\in\mathbb{N}$. In general, any neighborhood of $\Sigma_g$ cannot be saturated by $\mathcal{F}$ and the leaves are not an explicit cover of $\Sigma_{g,n}$. However, the following results study the topological of the intersection with a neighborhood of the $\mathcal{F}$-invariant compact surface. 

\begin{definition}
Let $\tau$ be a germen of transversal to a leaf $\mathcal{L}$ at a regular point $o\in\mathcal{L}$ and let  $\gamma=\gamma_{\sigma_1}^{\beta_1}\cdots\gamma_{\sigma_k}^{\beta_k}$ be a closed path, where $\{\gamma_j\}_{j\in\Lambda}$ are the generators of $\pi_1(\mathcal{L},o)$, and $\beta_l\in\mathbb{Z}, \sigma_l\in\Lambda$. Let $p\in\tau$, $l\in\mathbb{N}$ and the leaf $\mathcal{L}_p$ of $\mathcal{F}$ through $p$. We define $B(l,p,\tau)$ to be the union
$
\cup_{|\gamma|\leq l}\tilde{\gamma}_p(I),
$
where $\tilde{\gamma}_p:[0,1]\rightarrow\mathcal{L}_p$ is the lift of  $\gamma$ to $\mathcal{L}_p$ at $p$ and $|\gamma|=\sum_{j=1}^{k}|\beta_j|$. If $\tilde{\gamma}_p$ is well defined for each $\gamma$, such that $|\gamma|< l$, we call $B_l(p,\tau)$ the \textbf{graph ball} in $\mathcal{L}_p$ of radius $l$ and  center $p$.
\end{definition}

\begin{lemma}\label{lemma:ball}
Let $\mathcal{F}$ be a $\mathcal{C}^r$ differential singular foliation of real dimension two of a manifold $M^m$ with  $\mathcal{F}$-invariant compact Riemann surface  $\Sigma_g\subset M$, $m>2$.
Assume that  the set $\Sigma_g\cap\mathrm{Sing}(\mathcal{F})$ has cardinality $n$. If $\mathrm{Hol}(\Sigma_{g,n},\mathcal{F})$ is infinite. Then for each $N\in\mathbb{N}$ there is an embedding
 \[
    \varepsilon:B_N(\tilde{o})\rightarrow(\mathcal{L}_p,p),
 \]
 where $B_N(\tilde{o})\subset\mathrm{Cayley}(A_{g,n}^H)$ is as in Definition \ref{def:graphball} with $H=\ker(\mathrm{Hol}(\Sigma_{g,n},\mathcal{F}))$, and $\mathcal{L}_p$ is a leaf  through a regular point $p$ sufficiently close to $\Sigma_{g,n}$.
\end{lemma}
\begin{proof}
Let $\{\gamma_j\}_{j\in\Lambda}$ be the representations of the homotopy class of the canonical generators $\mathfrak{G}_{g,n}$ of $\pi_1(\Sigma_{g,n})$, whose image under $\varrho_H$ are nontrivial  generators of $A_{g,n}^H$. For any neighborhood of $o$ in $\tau$, there are points p, whose isotropy group $\mathrm{Iso}_{\mathrm{Hol}(\Sigma_{g,n},\mathcal{F})}$ is trivial, \cite[ Proposition 2.7]{godbillon98feuilletages}. Since $\mathrm{Hol}(\Sigma_{g,n},\mathcal{F})$ is infinite, for each $N\in\mathbb{N}$ exists a point $p\in\tau$ with trivial isotropy group sufficiently close to $o$ such that an holonomy map $h_\gamma$ on $\gamma\in\pi_1(\Sigma_{g,n})$ , satisfying $|\gamma|<N$, is well defined. Here $H$ denotes the kernel of $\mathrm{Hol}(\Sigma_{g,n},\mathcal{F})$(\ref{map:holonomyRepresentation}). Choose a vertex $\tilde{o}$ in the Cayley graph $\mathrm{Cayley}(A_{g,n}^H)$, and define a function 
\[
\begin{array}{lll}
\tilde{\epsilon}:\mathrm{Vertices}(B_N(p,\tau)) &\rightarrow & \mathrm{Vertices}(B_N(\tilde{o}))\\
\tilde{\gamma}_p(1) & \mapsto & \tilde{\gamma}_{\tilde{o}}(1)\\
\tilde{\gamma}_p(1) & \mapsto & \tilde{\gamma}_{\tilde{o}}(1),
\end{array}
\]
where $\tilde{\gamma}_{\tilde{o}}(1)$ is the end point of the lift of  $\gamma$ in $\mathrm{Cayley}(A_{g,n}^H)$ at the vertex $\tilde{o}$.

Suppose $\tilde{\epsilon}(\tilde{\gamma}_{p}(1))=\tilde{\epsilon}(\tilde{\gamma}_{p}'(1))=v_q$, then $\tilde{\gamma}_{\tilde{o}}\cdot\tilde{\gamma}_{v_q}'^{-1}(1)=\tilde{o}$. As $\mathrm{Iso}_{\mathrm{Hol}(\Sigma_{g,n},\mathcal{F})}(p)$ is trivial, thus $\gamma\cdot\gamma'^{-1}\in H$.  Since $B_N(p,\tau)$ is well defined and $\mathrm{Iso}_{\mathrm{Hol}(\Sigma_{g,n},\mathcal{F})}(p)$ is trivial, we have $h_{\gamma\cdot\gamma'^{-1}}$ is a trivial map and $h_{\gamma'}\circ h_{\gamma\cdot\gamma'^{-1}}=h_{\gamma}$. Hence $\tilde{\epsilon}$ is bijective. Since the lifts of the paths $\gamma_j$ at the points $\tilde{\gamma}_{p}(1)$,  $|\gamma|<N-1$, are edges of $B_N(p,\tau)$, we can extend $\tilde{\epsilon}$ to a graph isomorphism
\[
\begin{array}{lll}
\epsilon:B_N(p,\tau) &\rightarrow &\overline{B_N(\tilde{o})}\\
p & \mapsto & \tilde{o}\\
edge(\tilde{\gamma}_p(1),\widetilde{\gamma\cdot\gamma_j}^{\pm 1}_p(1)) & \mapsto & edge(\tilde{\gamma}_{\tilde{o}}(1),\widetilde{\gamma\cdot\gamma_j}^{\pm 1}_{\tilde{o}}(1)).
\end{array}
\] Hence $\epsilon^{-1}$ is a homeomorphism of graphs. Since $B_N(p,\tau)$ is compact in $M$, the homeomorphism $\varepsilon=\epsilon^{-1}$ is an embedding.

\end{proof}

We can now prove the Theorem \ref{THM:ReebType}, wich is an adaptation of the local stability theorem of Reeb for singular foliations.

\begin{proof}[Proof of Theorem \ref{THM:ReebType}]
By the lemma above, for each $N\in\mathbb{N}$ there exists a point $p\in\tau$ such that $B_N(p,\tau)$ is a graph ball in $\mathcal{L}_p$ homeomorphic to the graph ball $B_N(\tilde{0})$ in the graph $\mathrm{Cayley}(A_{g,n}^H)$ associated with $\Sigma_{g,n}^H$. Let $D_{g,k}(\delta)$ be a fundamental domain defined as in Definition \ref{def:graphball} such that it is open, simply connected and $\overline{D_{g,k}(\delta)}=\Sigma_{g,k}-\cup B_\delta(p_j)$, with $\{p_j\}_{j=1}^n=\mathrm{Sing}(\mathcal{F})\cap\Sigma_g$. We can choose $\delta>0$ and $o\in\tau'\subset\tau$ such that the lift $\hat{D}_{g,n}(\delta,q)$ through $q$ is well defined at each point $q\in\tau'$, \cite[Lemma 2,p.66]{camacho1979teoria}.

Therefore the embedding $\varepsilon$ of the lemma above extends to the interior of the surfaces
\[
\begin{array}{lll}
S_N(p,\tau)=\overline{\cup \hat{D}_{g,n}(\delta,\varepsilon(v))} &\quad & \varepsilon(v)\in Vertices(B_N(p,\tau)) \\
B_N(\delta,\tilde{o}) &\quad &v\in Vertices(B_N(\tilde{o})) .
\end{array}
\]
This implies that for  $p\in\tau$ sufficiently close to $o$ and trivial isotropy group, there exists an embedding \[\hat{\varepsilon}:B_N(\delta,\tilde{o})\rightarrow S_N(p,\tau)\subset\mathcal{L}_p.\] 
\end{proof}

\begin{corollary}
Under the assumptions of Theorem \ref{THM:ReebType}, if the regular cover $\Sigma_{g,n}^H$ has infinite genus; then, the foliation $\mathcal{F}$ has leaves with arbitrary genus nearly of $\Sigma_{g,n}$. Moreover, if $\mathrm{Hol}(\Sigma_{g,n},\mathcal{F})$ has a contracting map, then there are leaves of infinite genus.
\end{corollary}
\begin{proof}
If $\Sigma_{g,n}^H$ has infinite genus, there exists a minimal $N_0\in\mathbb{N}$ such that $B_{N_0}(\delta,\tilde{o})^{\circ}$ has genus different of zero, $g(S(N_0,\tilde{o})^{\circ})\neq0$. Since $\Sigma_{g,n}^H$ is regular, it follows that
\[g(B_{(a+1)N_0}(\delta,\tilde{o}))> g(B_{aN_0}(\delta,\tilde{o}))+g(B_{N_0}(\delta,\tilde{o}))\]
for any $a\in\mathbb{N}_{>0}$. Theorem \ref{THM:ReebType} now shows that $\mathcal{F}$ has leaves with arbitrary genus.
We now turn to the case when the group $\mathrm{Hol}(\Sigma_{g,n},\mathcal{F})$ has a contracting map $h$. Choose a point $p\in\tau$ whose isotropy group is trivial, then the points $h^m(p)$ has a trivial isotropy group. Thus, we can  embed  in $\mathcal L_p$ a surface $B_{N,\delta}(\tilde{o})$ for any $N\in\mathbb{N}$. It follows that $\mathcal L_p$ has unbounded genus.
\end{proof}

\section{Generic logarithmic foliations on $\mathbb{CP}^2$}

Let $M$ be a connected complex manifold of dimension at least 2, and let $D\subset M$ be a union of irreducible complex hypersurfaces $D_j$. A \emph{closed logarithmic 1-form} $\omega$ on $M$ with poles on $D$ is a meromorphic 1-form with the following property: for any $p\in M$ there exists a neighborhood $U$ of $p$ in $M$ such that $\omega|_U$ can be written as
 \begin{equation}\label{eq:logForm}
 \omega_{0}+\sum_{j=1}^{r}\lambda_{j}\frac{df_{j}}{f_{j}},
 \end{equation}
 where $\omega_0$ is a closed holomorphic 1-form on $U$, $\lambda_j\in\C^{\ast}$ and $f_j\in\mathcal{O}(U)$ , and $\{f_{j}=0\}$, $j=1,\ldots,r$, are the reduced equations of the irreducible components of $D\cap U$. The set $D$ is known as the polar divisor of $\omega$. The complex codimension one holomorphic foliation $\mathcal{F}$ of $M$ defined by $\omega$ is called \emph{logarithmic foliation}.
 
In particular, the holonomy group $\mathrm{Hol}(D_j^*,\mathcal{F})$,with $D_j^*=D_j-\mathrm{Sing}(\mathcal{F})$, associated with any irreducible component $D_j$ of the polar divisor $D$ of a logarithmic foliation $\mathcal{F}$ is abelian and linearizable (it is isomorphic to a subgroup of $\mathbb{C}^{\ast}$).  Moreover, if $M$ is simply connected then the holonomy group $\mathrm{Hol}(D_j^*,\mathcal{F})$ is a subgroup of the group generated by $\{\exp(2\pi i \frac{\lambda_k}{\lambda_j})\}$, where $\lambda_j$ is the residue of each irreducible component $D_j$ of $D$.

\begin{lemma}\label{lemma:genuslog}
Let $\mathcal{F}$ be a logarithmic foliation on a complex surfaces $M$ with polar divisor $D$. Assume that the holonomy group $\mathrm{Hol}(D_j^*,\mathcal{F})$ is infinite, with $D_j$ an irreducible component of $D$, and $\Sigma_{g,n}^H$ satisfies the condition (a) or (b) of Proposition \ref{Propt:AbelianCovers}, where $\Sigma_{g,n}=D_j^*$ and $H=\ker(\mathrm{Hol}(D_j^*,\mathcal{F}))$. Then $\mathcal{F}$ has leaves with infinitely many handles attached.
\end{lemma}
\begin{proof}
Let $\tau$ be a germ of transversal to the irreducible component $D_j$ of $D$ at a regular point $o\in D_j$ of $\mathcal{F}$.
Since $\mathrm{Hol}(D_j^*,\mathcal{F})$ is conjugate to a subgroup of $\C^*$, the isotropy group $\mathrm{Iso}_{\mathrm{Hol}(D_j^*,\mathcal{F})}(p)$ is trivial for each point $p\in\tau-o$. The group $\mathrm{Hol}(D_j^*,\mathcal{F})$ is infinite, then there is a map $h_\gamma\in\mathrm{Hol}(D_j^*, \mathcal{F})$ such that $\{h_{\gamma}^n(p)\}_{n\in\mathbb{N}}$ is contained in $\tau-o$, for a point $p\in\tau-o$ sufficiently close to $o$.

Let $\mathcal{L}$ be a leaf of $\mathcal{F}$ through $p$. By  Theorem \ref{THM:ReebType}, it suffices to show that $\Sigma_{g,n}^H$ has infinite genus, with $H={\mathrm{ker}\mathrm{Hol}(D_j^*,\mathcal{F})}$ and $n=|\mathrm{Sing}(\mathcal{F})\cap D_j|$, $g=g(D_j)$.
Since the regular cover $\Sigma_{g,n}^H$ satisfies the condition (a) or (b) of Proposition \ref{Propt:AbelianCovers}, we see that it has infinitely many handles attached, and the lemma follows.  
\end{proof}
\begin{remark}
Is possible to construct a logarithmic foliation on a complex surface $M$ with invariant compact Riemann surface $C$ and the singularities on it are all hyperbolic. Thus, for any open neighborhood $U\subset M$ of $C$ there is an open neighborhood $U'\subset U$ such that any leaf intersecting $U'$ is not contained in $U'$. 
Out of $U'$, Theorem \ref{THM:ReebType} can not say anything about the topology of the leaves.
\end{remark}

We can define a logarithmic foliation on the complex projective plane $\mathbb{CP}^2$ by a homogeneous closed logarithmic 1-form on $\mathbb{C}^3$ of the form
\[
\omega=F_1\cdots F_r\sum_{j=1}^r \lambda_j \frac{dF_j}{F_j},\quad \lambda_j\neq 0,\quad \sum_{j=1}^r d_j\lambda_j=0,
\] 
where each $F_j$ is homogeneous of degree $d_j$. On the complement $\mathbb{CP}^2-D$ we can define the following multivalued first integral 
\begin{equation*}
 F=\prod F_{j}^{\lambda_j}
\end{equation*}
which is a single valued map if we take its values in the quotient $\mathbb{C}/R$, with $R$ the product group generated by the numbers $\exp(2\pi i\lambda_j)$.
When all the ratios $\lambda_j/\lambda_i$ are rationals, it implies that the foliation has a rational first integral $F$ and all the leaves of $\mathcal{F}$ are an open subset of a complex algebraic curve. If the closure of a leaf $\mathcal{L}$ of $\mathcal{F}$ is not a complex algebraic curve, we call it \emph{non-algebraic} leaf. 

Note that for each singular point $p$ of $\mathcal{F}$ out of the polar divisor $D$, there is a simply connected open neighborhood $U$, such that $U\cap D=\emptyset$ and the restriction $\omega|_U$ is an exact 1-form. Thus, the number of separatrices through $p$ is finite. Now we call a leaf $\mathcal{L}$ of $\mathcal{F}$ \emph{generic}
if $\mathcal{L}$ is not an irreducible component of $D$ or a separatrix of any point in $\mathrm{Sing}(\mathcal{F})-D$. In general, the generic leaves are non-algebraic.
The following results describe the possible topological types of a non-algebraic leaf of logarithmic foliations on $\mathbb{CP}^2$.

\begin{lemma}\label{logarithmic ends}
Let $\mathcal{F}$ be a logarithmic foliation on $\mathbb{CP}^2$ with polar divisor $D$.
Let $\mathcal{L}$ be a non-algebraic leaf of $\mathcal F$. If $e$ is an end of $\mathcal{L}$ then
either locally the leaf $\mathcal{L}$ is a separatrix of a singularity of $\mathcal F$ on the complement of $D$, or
$\overline e \cap D \neq \emptyset$.
\end{lemma}
\begin{proof}
The divisor $D$ has at least two irreducible components.
Divide the set of irreducible components in two
sets, say $D_0$ and $D_{\infty}$. Let $F_0$ and $F_{\infty}$
be homogeneous polynomials of the same degree on $\mathbb C^3$ defining, respectively,
 $D_0$ and $D_{\infty}$. The quotient $\frac{F_0}{F_{\infty}}$
 defines a non-constant holomorphic map $ F : U \to \mathbb C^\ast$, where
 $U$ is the complement of $D$ in $\mathbb{CP}^2$.

From now on we assume that an end $e\in\mathcal{E}(\mathcal{L})$ is an open set of the nested sequence $\{P_k\}$ for a sufficiently large $k$. 
Let $K$ be a compact subset of $\mathcal{L}$. Let $e$ be an end of $\mathcal{L}$ contained in
a connected component of $\mathcal{L} - K$ such that the boundary $\partial_{\mathcal{L}}e$ in $\mathcal{L}$ is compact. The restriction of $F$ to $e$ is a holomorphic function $f: e \to \mathbb C$. If $f$ is constant, then $\mathcal{L}$ is an irreducible component of a fibre of the rational function and, hence, is algebraic which contradicts our assumptions.
So $f: e \to \mathbb C$ is a non-constant holomorphic function.

Let $V = f(e) \subset \mathbb C$ be the image of $f$. Since $f$ is holomorphic and non-constant $V$ is an open subset of $\mathbb C$. If it contains $\infty$ or $0$ in its closure, then, the lemma follows by continuity, the closure of the end $e$ intersects $D_{\infty}$ or $D_0$ respectively.

Assume from now on that $\infty,0 \notin \overline{f(e)}$.
Let $\mathcal G$ be the restriction of $\mathcal F$ to $U-K=U'$. The boundary $\partial e = \overline{e} - e$  in $U'$ is mapped by $f$ to $ \partial V$, the boundary of $V$. We point out that  $\partial e$ is invariant by $\mathcal G$, see \cite[Proposition 4.1.11]{candel2003foliations}. If $\partial e$ reduces to a point, the end in question
accumulates at one of the finitely many singularities of $\mathcal F$ in $U$. If instead the
boundary contains infinitely many points; then, it follows that $F(\partial e)$ contains infinitely many points. Therefore, $\mathcal G$ contains infinitely many leaves
contained in fibers of $F$. Thus $\mathcal F$ has infinite algebraic leaves. Jouanolou's Theorem implies that every leaf of $\mathcal{F}$ is algebraic, contradicting our assumptions again.

\end{proof}

We have so far describe the topological invariants of leaves of logarithmic foliations on $\mathbb{CP}^2$. We can prove our main result.

\begin{proof}[Proof of Theorem \ref{THM:LNMP2}]
Our proof starts with the observation that the polar divisor $D=\sum_{j=1}^r D_j$ has at least three irreducible components. Otherwise, the equation $\sum_{j=1}^2 d_j\lambda_j=0$ implies that $\lambda_1/\lambda_2$ is negative.
The assumption $\lambda_j/\lambda_l\notin\mathbb{R}_{<0}$ and the equations \[\sum_{j=1}d_j\frac{\lambda_j}{\lambda_l}=0\]
 implies that at least two ratios $\lambda_j/\lambda_l$ are complex numbers.
 
Our next claim is that the space of ends $\mathcal{E}(\mathcal{L})$ of the generic leaf $\mathcal{L}$ is a single point. 
Suppose the proposition is false. Then, we could find a generic leaf $\mathcal{L}$ and a compact subset $K$ of $\mathcal{L}$ such that the complement of $K$ in $\mathcal{L}$ has two connected components $e_1$ and $e_2$.
There is a neighbourhood $V$ of $D$ whose intersections with the level sets of $F:\mathbb{CP}^2\rightarrow\mathbb{C}/R$  are connected, \cite[Theorem B]{paul1997connectedness}. We can choose $V$ such that the intersection $V\cap K$ is empty. Since $\mathcal{L}$ is a generic leaf, Lemma \ref{logarithmic ends} shows that the intersections $E_j\cap V$, $j=1,2$, are not empty. Thus, $e_1$ and $e_2$ intersect in the connected set $V\cap\mathcal{L}$, a contradiction.
Therefore, the generic leaf $\mathcal{L}$ has a single end $e$. In particular, \cite[Proposition 2.3]{paul1997connectedness} implies that the closure of $e$ contains $D$. 
When $r>3$ or some $d_j>1$,  the fact that at least two ratios $\lambda_j/\lambda_l$ are complex numbers, the polar divisor $D$ is normal crossing and Lemma \ref{lemma:genuslog} imply that the generic leaf has infinitely many handles attached. Thus, the generic leaf $\mathcal{L}$ is homeomorphic to the Loch Ness Monster.   

Let $D$ be three lines in general position. Without loss of generality we assume that the  homogeneous equation in $\mathbb{C}^3$ are $l_1=\{x=0\},l_2=\{y=0\},l_3=\{z=0\}$ and $\lambda_1/\lambda_2\in\mathbb{C}$. Note that $\mathbb{CP}^2-D$ is biholomorphic to $(\mathbb{C}^*)^2$, thus we can define the universal covering 
\[\begin{array}{lll}
\rho:\mathbb{C}^{2} & \rightarrow  & \quad\quad\mathbb{CP}^{2}-D \\
 (x,y) & \mapsto &  (e^{2\pi ix}, e^{2\pi iy}).
\end{array}
\]
The pull-back $\rho^*\omega$ admits the following expression
 \[
 2\pi i(\lambda_1dx+\lambda_2dy)
 \]
which is a linear 1-form on $\mathbb{C}^{2}$. Thus, the leaves $\mathcal{L}^*$ of $\rho^*\mathcal{F}$ are complex lines. Since $\lambda_1/\lambda_2\in\mathbb{C}$, the restriction $\rho|_{\mathcal{L}^*}$ is a biholomorphis on its image, which are the generic leaves of $\mathcal{F}$, and the proof is complete.
\end{proof}
\begin{example}
Here we generalize an example given by Cerveau in \cite[Subsection 2.12]{cerveau2013quelques}. The homogeneous closed logarithmic 1-form 
\[\omega=\prod_{j=1}^r x_j\sum_{j=1}^r\lambda_j\frac{dx_j}{x_j},\quad\sum_{j=1}^r\lambda_j=0\] 
in $\mathbb{C}^{n+1}$ defines a logarithmic foliation $\mathcal{F}$ on $\mathbb{CP}^n$. Assume $n+1\geq r>3$ and the ratios of residues $\lambda_j$ are not negative real numbers. An analysis similar to that in the proof of Theorem \ref{THM:LNMP2} shows that the leaves different from the irreducible components of the polar divisor $\sum_j H_j$ of $\omega$ are biholomorphic to $\mathbb{C}^{n-1}$, with $H_j=\{x_j=0\}$. Consider a plane projective $P\subset\mathbb{CP}^n$ through the point $o=[1:0:\cdots:0]$, whose intersections with the hyperplanes out of $o$ are general. Thus the restriction $\mathcal{F}'=\mathcal{F}|_{P}$ is a homogeneous foliation, and by the assumptions on the ratios the generic leaf of $\mathcal{F}'$ is homeomorphic to the Loch Ness Monster. Moreover, if $P$ is general respect all the hyperplanes $H_j$ then Theorem \ref{THM:LNMP2} shows that the generic leaf of $\mathcal{F}'$ is also homeomorphic to the Loch Ness Monster.
\end{example}
\begin{remark}
The Example \ref{exp:logFol1} gives a logarithmic foliation with generic leaf homeomorphic to a plane without an infinite discrete set of points. In this case, any element of the space of ends of a generic leaf accumulates in the polar divisor. Thus, it is not enough that a holomorphic foliation has an invariant compact Riemann surface to ensure that the foliation has leaves with a non-trivial genus.
\end{remark}
\begin{remark}
An essential fact in the E. Ghys theorem \cite{ghys1995topologie} is that he can decide the space of ends for the generic leaf, this still unknown for almost every class of holomorphic foliations on compact complex surfaces.
\end{remark}

\bibliographystyle{alpha}

\end{document}